\documentclass[reqno]{amsart}
\usepackage{amssymb,amsmath,hyperref}
\usepackage{amsrefs, float}
\usepackage[foot]{amsaddr}
\usepackage{bbold,stackrel, multicol}
\usepackage{mathrsfs}

\newcommand{\Z}{{\textsf{\textup{Z}}}}

\newtheorem{thm}{Theorem}

\newtheorem{cor}[thm]{Corollary}
\newtheorem{defi}[thm]{Definition}
\newtheorem{rem}[thm]{Remark}
\newtheorem{nota}[thm]{Notation}
\newtheorem{exa}[thm]{Example}

\newtheorem{princ}[thm]{Principle}

\newtheorem{ack}[thm]{Acknowledgement}

\newtheorem{temp}[thm]{Template}
\newtheorem{theorem}[thm]{Theorem}

\newcommand\be{\begin{equation}}
\newcommand\ee{\end{equation}} 

\usepackage{amsmath,amsfonts} 

\usepackage[applemac]{inputenc}

\def\bdefi{\begin{defi}}
\def\edefi{\end{defi}}
\def\bnota{\begin{nota}\rm}
\def\enota{\end{nota}}
\def\FIVE{\Pi_{1}^{1}\text{-\textup{\textsf{CA}}}_{0}}
\def\SIX{\Pi_{2}^{1}\text{-\textsf{\textup{CA}}}_{0}}

\def\ATR{\textup{\textsf{ATR}}}

\def\ZFC{\textup{\textsf{ZFC}}}
\def\ZF{\textup{\textsf{ZF}}}

\def\osc{\textup{\textsf{osc}}}

\def\RCA{\textup{\textsf{RCA}}}
\def\({\textup{(}}
\def\){\textup{)}}

\def\RCAo{\textup{\textsf{RCA}}_{0}^{\omega}}
\def\ACAo{\textup{\textsf{ACA}}_{0}^{\omega}}

\def\WKL{\textup{\textsf{WKL}}}

\def\bye{\end{document}}
\def\N{{\mathbb  N}}
\def\Q{{\mathbb  Q}}
\def\R{{\mathbb  R}}

\def\SS{\textup{\textsf{S}}}

\def\di{\rightarrow}

\def\asa{\leftrightarrow}

\def\ACA{\textup{\textsf{ACA}}}

\def\QFAC{\textup{\textsf{QF-AC}}}

\def\PHP{\textup{\textsf{PHP}}}

\def\osc{\textup{\textsf{osc}}}

\def\alt{\textup{\textsf{alt}}}

\def\enum{\textup{\textsf{enum}}}

\def\WSAC{\textup{\textsf{weak-$\Sigma_{1}^{1}$-AC$_{0}$}}}
\def\SAC{\textup{\textsf{$\Sigma_{1}^{1}$-AC$_{0}$}}}

\def\NIN{\textup{\textsf{NIN}}}

\def\BCT{\textup{\textsf{BCT}}}

\def\open{\textup{\textsf{open}}}

\def\IND{\textup{\textsf{IND}}}

\def\eps{\varepsilon}

\newcommand{\F}{{\bf F}}

\usepackage{graphicx}
\usepackage{tikz}
\usetikzlibrary{matrix}
\usepackage{comment,tikz-cd}

\numberwithin{equation}{section}
\numberwithin{thm}{section}

\begin{document}
\title{On two recent extensions of the Big Five of Reverse Mathematics}
\author{Dag Normann}
\address{Department of Mathematics, The University of Oslo, Norway, P.O. Box 1053, Blindern N-0316 Oslo}
\email{dnormann@math.uio.no}
\author{Sam Sanders}
\address{Department of Philosophy II, RUB Bochum, Germany}
\email{sasander@me.com}
\begin{abstract}
The program \emph{Reverse Mathematics} in the foundations of mathematics seeks to identify the 
minimal axioms required to prove theorems of ordinary mathematics.  One always assumes the \emph{base theory}, a logical system embodying computable mathematics.  
As it turns out, many (most?) theorems are either provable in said base theory, or equivalent to one of four logical systems, collectively called the \emph{Big Five}. 
This paper provides an overview of two recent extensions of the Big Five, working in Kohlenbach's higher-order framework.  On one hand, we obtain a large number of equivalences between the second-order Big Five and 
third-order theorems of real analysis dealing with possibly discontinuous functions.
On the other hand, we identify four new `Big' systems, i.e.\ boasting many equivalences over the base theory, namely \emph{the uncountability of the reals}, the \emph{Jordan decomposition theorem}, the \emph{Baire category theorem}, and Tao's \emph{pigeon hole principle} for the Lebesgue measure.  We discuss a connection to hyperarithmetical analysis, completing the picture.   
\end{abstract}


\maketitle
\thispagestyle{empty}

\section{Introduction and preliminares}\label{intro}
We provide an overview of the results in \cites{dagsamXIV, dagsamXI,dagsamX, samBIG, samBIG2, samBIG3,samBIG4} with a focus on two recent extensions of the \emph{Big Five} of \emph{Reverse Mathematics} (RM for short).
We will fist introduce the latter italicised notions and then sketch our contributions. 

\smallskip

First of all, the aim of RM is to find the minimal axioms needed to prove a given theorem of ordinary, i.e.\ non set-theoretic mathematics.
Generally, the minimal axioms are also {equivalent} to the theorem at hand, as observed by Friedman.  
\begin{quote}
When a theorem is proved from the right axioms, the axioms can be proved from the theorem. (\cite{fried})
\end{quote}
The RM-program was founded by Friedman (\cites{fried,fried2}) and developed extensively by Simpson and others (\cites{simpson2, simpson1}).  
The original textbook is \cite{simpson2} with a recent textbook \cite{damu}.  An introduction to RM for the proverbial mathematician-in-the-street is \cite{stillebron}.  
Now, the \emph{Big Five phenomenon} is a central topic in RM, as follows. 
\begin{quote}
[...] we would still claim that the great majority of the theorems from classical mathematics are equivalent to one of the big five. This phenomenon is still quite striking. Though we have some sense of why this phenomenon occurs, we really do not have a clear explanation for it, let alone a strictly logical or mathematical reason for it. The way I view it, gaining a greater understanding of this phenomenon is currently one of the driving questions behind reverse mathematics. (see \cite{montahue}*{p.\ 432})
\end{quote}
As discussed in minute detail in \cite{dsliceke, damu}, the study of combinatorics in RM has yielded many relatively natural principles that are classified \emph{outside} of the aforementioned Big Five classification.  
The collection of these exceptional principles has been dubbed the \emph{RM zoo} and comes with a convenient computer tool (\cite{damirzoo}).

\smallskip

Secondly, the previous takes place in the language of second-order arithmetic, which cannot directly express third-order notions like sets of reals or functions on the reals.  
For this reason, the focus of second-order RM has been on countable objects and higher-order objects that come with a countable representation or `code'.  
In our project, and this paper, we study the logical and computational properties of the uncountable, with a minimum of codes/representation 
For this reason, we work in Kohlenbach's higher-order RM (\cite{kohlenbach2}) where we stress that the base theory $\RCAo$ is a conservative extension of the usual base theory $\RCA_{0}$ of RM.  
The real numbers have the same definition in $\RCA_{0}$ and $\RCAo$ and $\R\di \R$-functions are $\N^{\N}\di \N^{\N}$-functions that respect equality `$=_{\R}$'.

\smallskip

Our first extension of the Big Five is based on the observation that while equivalences exist involving the Big Five and theorems from real analysis, 
the latter are mostly limited to continuous functions (see e.g.\ \cite{simpson2}*{IV.2}).  In a nutshell, the goal of \cites{dagsamXIV, samBIG3} is to obtain \emph{many} equivalences involving
the Big Five on one hand, and theorems concerning possibly \emph{discontinuous} functions on the other hand.  
We discuss these results in more detail in Section \ref{fluky}.  

\smallskip

Our second extension of the Big Five is based on an important observation made in \cite{dagsamXIV}: while many theorems of real analysis are equivalent to the Big Five, certain \emph{slight} variations or generalisations 
are not provable in the Big Five and much stronger systems.   However, the principles that defy a Big Five-classification can generally be classified as equivalent to one of the following principles, again working in higher-order RM.
\begin{itemize}
\item The uncountability of $\R$.
\item The enumeration principle: a countable set of reals can be enumerated.
\item The Baire category theorem.
\item  Tao's pigeon hole principle for the Lebesgue measure.
\end{itemize}
The exact meaning of `countable' is discussed in Section \ref{horsies}, but we can reveal that injections to $\N$ and especially Borel's \emph{height functions} (\cite{opborrelen3, opborrelen4, opborrelen5}) play a central goal.  
The goal of \cites{dagsamXI, samBIG, samBIG2, samBIG3} is to formulate many equivalences involving these four new `Big' systems. 
We discuss these results in more detail in Section \ref{fluky2}. 

\smallskip

Together, our new extensions of the Big Five shine a new light upon the aforementioned coding practise of second-order RM in which open sets and continuous functions are `on equal footing'. 
Indeed, a code for an open set can be effectively converted to a code for a continuous function, and vice versa, over the base theory $\RCA_{0}$ (\cite{simpson2}*{II.7.1}).  
Now, by Theorem \ref{nudge}, a weak system, namely $\RCAo+\WKL_{0}$, suffices to prove that third-order continuous functions on the unit interval have codes.  
By contrast, that a third-order open set has a code is not provable from the Big Five and much stronger systems (Theorem \ref{temp3}).  
In particular, continuous functions and open sets are \emph{not} on equal footing in higher-order arithmetic.  
Moreover, Figure \ref{xxx} shows the relation between some of the above principles; the assumption that open sets are given by RM-codes completely erases the RM of the aforementioned new Big systems. 

\begin{figure}[H]
\begin{tikzpicture}
  \matrix (m) [matrix of math nodes,row sep=3em,column sep=4em,minimum width=2em]
  { ~ & \begin{array}{c}\open \text{ = an open set of}\\ \text{reals has an RM-code.}\end{array}& ~\\
\begin{array}{c} \enum \text{= a countable} \\ \text{ set of reals can} \\ \text{be enumerated.}\end{array}&   \begin{array}{c} \BCT_{[0,1]} \text{= the Baire } \\ \text{category theorem} \\ \text{for the unit interval.}\end{array}        &\begin{array}{c} \PHP_{[0,1]} \text{= pigeon}\\\text{hole principle for} \\ \text{Lebesgue measure} \end{array}  \\
        ~          & \begin{array}{c}  \NIN_{[0,1]} \textup{= there is no }\\ \text{ injection from $[0,1]$ to $\N$}\end{array} &           ~ \\};
  \path[-stealth]
    (m-1-2) edge node [left] {} (m-2-1)
    (m-1-2) edge node [left] {} (m-2-2)
        (m-1-2) edge node [left] {} (m-2-3)
    (m-2-1) edge node [left] {} (m-3-2)
    (m-2-3) edge node [left] {} (m-3-2)
    (m-3-2) edge node [left] {} (m-3-2)
        (m-2-2) edge node [left] {} (m-3-2)
    (m-2-1) edge[bend left=10] node [above] {?} (m-2-2)
    (m-2-2) edge[bend right=-10] node [below] {?} (m-2-1)
    (m-2-2) edge[bend left=10] node [above] {?} (m-2-3)
    (m-2-3) edge[bend right=-10] node [below] {?} (m-2-2);
\end{tikzpicture}
\caption{Some relations among our principles}
\label{xxx}
\end{figure}
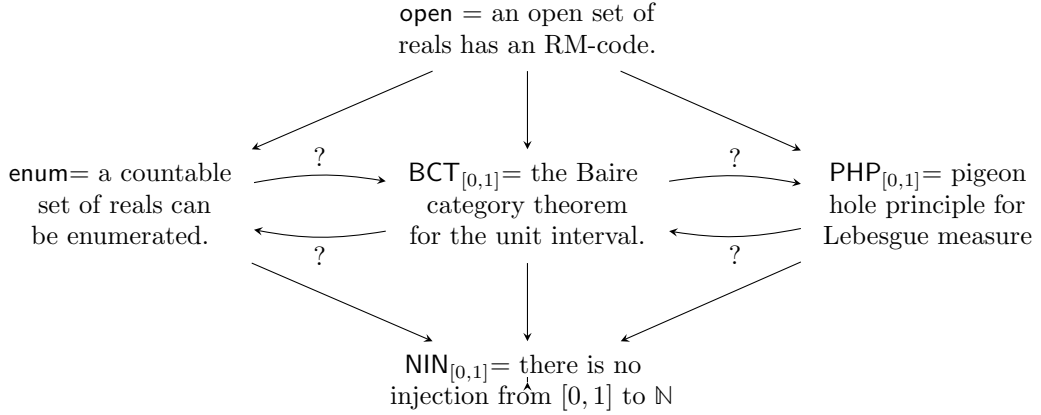
Next, regarding the base theory, we sometimes work over $\RCAo+\QFAC^{0,1}$, where the latter is the following fragment of the Axiom of (countable) Choice. 
\begin{princ}[$\QFAC^{0,1}$] Let $\varphi$ be any quantifier-free formula with arbitrary parameters and such that $(\forall n\in \N)(\exists f\in \N^{\N})\varphi(n ,f)$.  Then there is a sequence $(f_{n})_{n\in \N}$ in Baire space such that $(\forall n\in \N)\varphi(n, f_{n})$.
\end{princ}
\noindent
We provide the following two-fold motivation for the base theory $\RCAo+\QFAC^{0,1}$.
\begin{itemize}
\item The system $\ZF$ cannot prove the pointwise equivalence between sequential and `epsilon-delta' continuity, while $\RCAo+\QFAC^{0,1}$ can (see \cite{kohlenbach2}). 
\item Many results proved over $\RCAo+\QFAC^{0,1}$ do not go through over $\RCAo$ (\cites{dagsamV, dagsamVII}), like the equivalence between $\WKL_{0}$ and the Heine-Borel theorem for sequences of open sets without representations, as in Definition \ref{charz}.   
\end{itemize}
Now, $\QFAC^{0,1}$ also yields a conservative extension of $\SAC$, a system of hyperarithmetical analysis (\cite{hunterphd}*{\S2}).  
As it turns out, many third-order theorems of real analysis exist \emph{in the range of hyperarithmetical analysis}.  
The latter means that these theorems are between conservative extensions of systems of hyperarithmetical analysis, generally $\SAC$ and $\WSAC$ (see \cite{simpson2} for the latter).  
We sketch these results in Section \ref{HYPsec} but can already reveal that bijections to $\N$ play an important role.   In this way, a satisfying and complete picture comes to the fore.   

\smallskip

Finally, there are \emph{many} function classes that are part of real analysis.  
We have collected the relevant definitions in Section \ref{FCS}, most of which are studied in \cite{dagsamXIV} and elsewhere. 
We would like to stress that the usual hierarchy of function classes, provable in say $\ZFC$, can look \emph{very} different
in weak systems.  For instance, most of the function classes from Section \ref{FCS} do not contain totally discontinuous functions. 
Yet it is consistent with rather strong systems that there are totally discontinuous functions in those function classes.  
An explicit example is the function $h$ from \eqref{mopi} in case the unit interval is countable, i.e.\ there us an injection from the latter to $\N$.

\section{The Biggest Five of Reverse Mathematics}\label{fluky}
In this section, we introduce a recent extension of the second-order Big Five based on real analysis from \cites{dagsamXIV, samBIG,samBIG2, samBIG3}.  
In particular, we discuss equivalences between the \emph{second-order} Big Five and \emph{third-order} theorems of real analysis dealing with possibly discontinuous functions.  
The reader interested in foundational issues will note that the second- and third-order worlds are apparently intimately intertwined, much more than the first- and second-order ones are. 

\smallskip

We shall make use of the mainstream definition (of continuity, Riemann integration, \dots), i.e.\ no coding is used unless explicitly stated otherwise.   
The RM of $\WKL_{0}$ is developed in some detail (Section \ref{sketch}) while we merely sketch the results for the other Big Five (Section \ref{othersketch}).  

\subsection{Equivalences involving weak K\"onig's lemma}\label{sketch}
First of all, the second-order RM of real analysis has mostly focused on properties of \emph{continuous} functions as in Theorem \ref{temp0}.  
We assume familiarity with the associated coding (\cite{simpson2}*{II.6.1}).
\begin{thm}[$\RCA_{0}$,  \cite{simpson2}*{IV.2}]\label{temp0}
The following are equivalent to $\WKL_{0}$.  
\begin{itemize}
\item Any \(code for\) a continuous function $f:[0,1]\di \R$ is bounded above. 
\item Any \(code for\) a continuous function $f:[0,1]\di [0,1]$ has a supremum. 
\item Any \(code for\) a continuous function $f:[0,1]\di \R$ with a supremum, attains its maximum.
\end{itemize}
\end{thm}
Now, the first step is to remove the coding from Theorem \ref{temp0}.  
As it turns out, only relatively weak axioms are needed to show that total second-order codes denote third-order functions that are continuous, and vice versa. 
\begin{thm}[\cite{dagsamXIV}]\label{nudge}~
\begin{itemize}
\item \($\RCAo$\) For a second-order code $\Phi$ of a continuous function on $\R$, there is a third-order $f:\R\di \R$ such that $f(x)$ equals the value of $\Phi$ at any $x\in \R$.
\item \($\RCAo+\WKL_{0}$\) For a third-order $f:[0,1]\di \R$ that is continuous, there is a second-order code $\Phi$ such that $f(x)$ equals the value of $\Phi$ at any $x\in [0,1]$.
\end{itemize}
\end{thm}
\begin{proof}
For the first item, the base theory $\RCAo$ includes $\QFAC^{1,0}$ defined as follows:
\be\label{lebber}
(\forall f\in \N^{\N})(\exists n\in \N)\varphi(f, n)\di (\exists G:\N^{\N}\di \N)(\forall f\in \N^{\N})\varphi(f, G(f))
\ee
for quantifier-free $\varphi$.
Up to coding, the statement `for any $x\in \R$, $\Phi$ is total at $x$' has the same form as the antecedent of $\QFAC^{1, 0}$.
Applying the latter, the resulting choice function yields the required third-order function.

\smallskip

For the second item, Kohlenbach proves the associated result for continuous functions on Cantor space (\cite{kohlenbach4}*{\S4.4}), based on a construction due to Dag Normann.  
This construction is adapted in \cite{dagsamXIV} to the unit interval.  Simply put, the aformentioned construction starts from the usual `epsilon-delta' definition 
of continuity and uses $\WKL_{0}$ to show that the innermost universal formula -involving a quantifier over $[0,1]$ or $2^{\N}$- is equivalent to a \emph{quantifier-free} formula.  
One then obtains a modulus of (uniform) continuity, which readily yields a second-order code. 
\end{proof}
\begin{cor}[$\RCAo$]\label{temp00}
The following are equivalent to $\WKL_{0}$.  
\begin{itemize}
\item Any third-order function $f:[0,1]\di \R$ that is continuous, is bounded above. 
\item Any third-order function $f:[0,1]\di \R$ that is continuous, has a supremum.
\item Any third-order function $f:[0,1]\di \R$ that is continuous and has a supremum, attains its maximum. 
\end{itemize}
\end{cor}
\begin{proof}
Use the theorem to obtain codes and apply Theorem \ref{temp0}.
\end{proof}
The next crucial step is the realisation that properties like boundedness also hold for large classes of \emph{discontinuous} functions.  
In particular, it is now a natural question for which function classes $\Gamma, \Xi, \Omega, \dots$ of real analysis Theorem \ref{temp1} holds.
\begin{thm}[$\RCAo$ or $\RCAo+\QFAC^{0,1}$]\label{temp1}
The following are equivalent to $\WKL_{0}$.  
\begin{itemize}
\item Any $f:[0,1]\di \R$ in $\Gamma$ is bounded above. 
\item Any $f:[0,1]\di [0,1]$ in $\Xi$ has a supremum\footnote{To be absolutely clear, we assume the existence of a `supremum operator' $\Phi:\Q^{2}\di \R$ such that $\Phi(p, q)=\sup_{x\in [p, q]}f(x)$ for all $p, q\in [0,1]\cap \Q$.\label{footie}}. 
\item Any $f:[0,1]\di \R$ in $\Omega$ such that $\sup_{x\in [0,1]}f(x)$ exists, attains a maximum.
\end{itemize}
\end{thm}
As explored in \cites{dagsamXIV, samBIG, samBIG2, samBIG3}, Theorem \ref{temp1} holds for the following function classes , the definitions of which can be found in the mainstream literature or Section \ref{FCS}.
\begin{itemize}
\item $\Gamma$ consists of functions that are either: regulated, regulated and $U_{0}$, regulated and continuous ae, bounded Waterman variation, cadlag, usco, $C\cup BV$, regulated and effectively Baire $n$, Baire 1 and Darboux, Baire 1 and regulated, or regulated and quasi-continuous.
\item $\Xi$ is either: quasi-continuous, Baire 1, quasi-continuous and regulated, cadlag, Baire 1 and Darboux, or Baire 1 and usco. 
\item $\Omega$ is either: usco, usco and regulated, usco and Baire 1, cadlag and usco, or usco and quasi-continuous. 
\end{itemize}
We point out that there are \emph{many}\footnote{If $\mathfrak{c}$ is the cardinality of $\R$, there are $2^{\mathfrak{c}}$ non-measurable quasi-continuous $[0,1]\di \R$-functions and $2^{\mathfrak{c}}$ measurable  quasi-continuous $[0,1]\di [0,1]$-functions (see \cite{holaseg})} quasi-continuous functions and that this notion already goes back to Baire and Volterra (\cite{beren2, kemphaan}).  Moreover, there are many function classes in real analysis, many of which should yield generalisations of Theorem \ref{temp1}.  

\smallskip

Another interesting example is \emph{Cousin's lemma} (\cite{cousin1}), defined as follows.
\begin{center}
\emph{For $\Psi:[0,1]\di \R^{+}$, the covering $\cup_{x\in [0,1]}B(x, \Psi(x))$ of $[0,1]$ has a finite sub-covering, i.e.\ there are $x_{0}, \dots, x_{k}\in [0,1]$ where $\cup_{i\leq k}B(x_{i}, \Psi(x_{i}))$ covers $[0,1]$.}
\end{center}
Now, $\WKL_{0}$ is equivalent to Cousin's lemma for (codes of) continuous functions (\cite{basket, basket2}). 
The latter result readily generalises as follows.
\begin{thm}[$\RCAo$]\label{temp2}
The principle $\WKL_{0}$ is equivalent to Cousin's lemma restricted to either: lsco, lsco and Baire 1, lsco and effectively Baire $n+2$, lsco and continuous ae, lsco and pointwise discontinuous, lsco and not everywhere discontinuous, quasi-continuous, cadlag, regulated and  $U_{0}$, or Baire 1.  
\end{thm}
\begin{proof}
Many instances follow by Theorem \ref{temp1}.  The full proof is in \cite{dagsamXIV}.
\end{proof}
The above results merely constitute examples: many equivalences exist for $\WKL_{0}$ involving Bernstein polynomial approximation and Riemann integration of discontinuous functions (\cite{samBIG3, samBIG2}).  
Moreover, the focus of \cite{dagsamXIV} was on the Big Five beyond $\RCA_{0}$.  Nonetheless, plenty of real analysis can be established in $\RCAo$.
\begin{thm}[$\RCAo$]\label{temp4}~
\begin{itemize}
\item For cadlag $f:[0,1]\di \R$, the discontinuity points can be enumerated.
\item For cadlag $f:[0,1]\di \R$, the continuity points of $f$ are dense.
\item For cadlag $f:[0,1]\di \R$, there is a point of continuity. 
\end{itemize}
We may replace `cadlag' by `quasi-continuity' in the final two items.
\end{thm}
\begin{proof}
We only need to prove the first item.  By \cite{dagsamXIV}*{Theorem 2.16}, one can enumerate the jump discontinuities of a regulated function in $\RCAo$.  
Clearly, cadlag functions only have such discontinuities. 
The results for quasi-continuous functions are immediate by \cite{samBIG2}*{Theorem 2.19}.
\end{proof}
Many of the above results were proved via a `trick' or `shortcut' based on the law of excluded middle (LEM), as discussed in Remark \ref{LEM}.
The reader will give no second thought to the observation that all systems of RM make use of classical logic.  
Nonetheless, the following `special case' has been met with surprise.
\begin{rem}[The LEM trick]\label{LEM}\rm
Our starting point is Kleene's arithmetical quantifier $(\exists^{2})$, defined as follows:
\be\label{muk}\tag{$\exists^{2}$}
(\exists E:\N^{\N}\di \{0,1\})(\forall f \in \N^{\N})\big[(\exists n\in \N)(f(n)=0) \asa E(f)=0    \big].
\ee
The system $\ACAo\equiv\RCAo+(\exists^{2})$ is a conservative extension of $\ACA_{0}$ by \cite{hunterphd}*{Theorem 2.5}.
By \cite{kohlenbach2}*{Prop.\ 3.12}, $(\exists^{2})$ is equivalent over $\RCAo$ to the statement 
\begin{center}
\emph{There exists an $\R\di \R$-function that is not continuous.}
\end{center}
Clearly, $\neg(\exists^{2})$ is then equivalent to \emph{Brouwer's theorem}, i.e.\ the statement that all $\R\di \R$-functions are continuous. 
Now, if we wish to prove a given statement $T$ of real analysis about possibly discontinuous functions in $\RCAo+\WKL_{0}$, we may invoke the law of excluded middle as in $(\exists^{2})\vee \neg(\exists^{2})$.  
We can then split the proof of $T$ in two cases: one assuming $\neg(\exists^{2})$ and one assuming $(\exists^{2})$.  
In the latter case, since $(\exists^{2})\di\ACA_{0}$, we have access to much more powerful tools (than just $\WKL_{0}$).  
In the former case, since $\neg(\exists^{2})$ implies that all functions are continuous, we only need to establish $T$ restricted to the special case of continuous functions.  
Moreover, we can invoke Theorem \ref{nudge} to provide codes for all (continuous) functions.  After that, we can use the second-order RM literature to establish $T$ restricted to codes for continuous functions, and hence $T$. 
To be absolutely clear, the `LEM trick' is the above splitting of proofs based on $(\exists^{2})\vee \neg(\exists^{2})$.
\end{rem}
Finally, we have extended the RM of $\WKL_{0}$ by numerous basic theorems of real analysis about possibly discontinuous functions.  
One naturally wonders how far this extension can be pushed.  Theorem \ref{temp3} shows that \emph{slight} variations and generalisations of the items in Theorems \ref{temp1}-\ref{temp4}, are no longer provable from the Big Five and much stronger\footnote{The system $\Z_{2}^{\omega}$ consists of $\RCAo$ plus the axioms $(\SS_{k}^{2})$ for any $k$; the latter axiom states the existence of a functional $\SS^{2}_{k}$ that decides $\Pi_{k}^{1}$-formulas in Kleene normal form.  By \cite{hunterphd}*{Cor.\ 2.6}, $\Z_{2}^{\omega}$ is a conservative extension of $\Z_{2}$.\label{omega}} systems like $\Z_{2}^{\omega}$. 
\begin{thm}[$\ACAo$]\label{temp3}
The following principles imply $\NIN_{[0,1]}$, i.e.\ the statement that there is no injection from $[0,1 ]$ to $\N$.
\begin{itemize}
\item A regulated $f:[0,1]\di \R$ is Baire 1 \(or: not totally discontinuous\).
\item A regulated \(or $BV$\) $f:[0,1]\di \R$ has a supremum$^{\ref{footie}}$.
\item An usco $f:[0,1]\di \R$ is Baire 1 \(or: not totally discontinuous\). 
\item \($\open$\) An open set $O\subset [0,1]$ can be represented by an RM-code. 
\item A fragmented $f:[0,1]\di \R$ is Baire 1 \(or: not totally discontinuous\). 
\item A Baire 1$^{*}$ $f:[0,1]\di \R$ is Baire 1 \(or: not totally discontinuous\). 
\item A $B$-measurable function of first class $f:[0,1]\di \R$ is Baire 1 \(or: not totally discontinuous\). 
\item An usco $f:[0,1]\di [0,1]$ has a supremum$^{\ref{footie}}$. 
\item A cliquish $f:[0,1]\di [0,1]$ has a supremum$^{\ref{footie}}$. 
\item A cliquish $f:[0,1]\di [0,1]$ is not totally discontinuous. 
\end{itemize}
The system $\Z_{2}^{\omega}+\QFAC^{0,1}$ cannot prove $\NIN_{[0,1]}$.
\end{thm}
\begin{proof}
The final sentence is proved in \cite{dagsamX}. 
The other results have straightforward proofs-by-contradiction as follows: let $Y:[0,1]\di \N$ be injective and define $f(x):=\frac{1}{2^{Y(x)+5}}$.    
This function satisfies the conditions of the items in the theorem, but is totally discontinuous. Moreover, using the supremum operator, the usual interval halving technique yields an enumeration of $[0,1]$, i.e.\ a contradiction. 
Indeed, we can decide $\sup_{x\in [0, \frac12]}f(x)<\sup_{x\in [0,1]}f(x)$ using $(\exists^{2})$ and thus obtain the first bit of the binary expansion of $x_{0}\in [0,1]$ such that $f(x_{0})=\sup_{x\in [0,1]}f(x)$.  
Continuing in this way, one obtains the binary expansion of $x_{0}$; we then repeat the same process for $f$ capped below $f(x_{0})$, ultimately enumerating all of $[0,1]$ since $Y$ is an injection. 
Finally, bounded Baire~1 functions have a supremum by Theorem \ref{temp1}, while closed sets have usco characteristic functions, i.e.\ the supremum of the latter readily yields an RM-code. 
\end{proof}
In conclusion, we have established the equivalence between $\WKL_{0}$ and a number of theorems from real analysis.  
We have also shown that slight generalisations or variations are no longer provable in the Big Five and much stronger systems.  
A reasonable explanation of this phenomenon may be found in Section \ref{seor}. 

\subsection{Equivalences for the other Big Five}\label{othersketch}

\subsubsection{Arithmetical comprehension}
We discuss equivalences between arithmetical comprehension and restrictions of the Jordan decomposition theorem.
The full version of the latter is part of the RM of the enumeration principle (Section \ref{Esec}) and hence not provable in $\Z_{2}^{\omega}+\QFAC^{0,1}$ from Footnote \ref{omega}.  

\smallskip

First of all, the fundamental theorem about $BV$-functions (see e.g.\  \cite{jordel}*{p.\ 229}) is as follows.
\begin{thm}[Jordan decomposition theorem]\label{drd}
A function $f : [0, 1] \di \R$ of bounded variation is the difference of two non-decreasing functions $g, h:[0,1]\di \R$.
\end{thm}
Theorem \ref{drd} has been studied extensively via second-order representations in e.g.\ \cites{groeneberg, kreupel, nieyo, verzengend}.
We study certain restrictions as follows. 
\begin{thm}[$\RCAo+\QFAC^{0,1}$]\label{XZ} The following are equivalent to $\ACA_{0}$.
\begin{itemize}
\item The Jordan decomposition theorem for cadlag $BV$-functions.
\item The Jordan decomposition theorem for $BV$-functions satisfying the equality $f(x)=\frac{f(x+)+f(x-)}{2}$ for $x\in (0,1)$.
\item The Jordan decomposition theorem for quasi-continuous $BV$-functions.
\item A number of variations of the Jordan decomposition theorem involving Bernstein polynomials \(see \cite{samBIG3}*{Theorem 2.4}\).
\end{itemize}
\end{thm}
\begin{proof}
By \cite{dagsamXIV}*{Theorem 2.16}, the jump discontinuities of a regulated functions can be enumerated.
The functions at hand only have such discontinuities, i.e.\ we can enumerate the set of discontinuity points.  
With access to the latter, the standard proof of the Jordan decomposition theorem goes through (see \cite{voordedorst}). 
\end{proof}
Secondly, the RM of $\ACA_{0}$ involves basic analysis, like \cite{simpson2}*{IV.2.11 and III.2.2}.
The following generalisations of the latter are equivalent to $\ACA_{0}$ over $\RCAo$.
\begin{itemize}
\item Let $F:C\di \R$ be cadlag (or: quasi-continuous) where $C\subset [0,1]$ is an RM-closed set.  Then $\sup_{x\in C}F(x)$ exists. 
\item Let $F:C\di \R$ be cadlag (or: quasi-continuous) and usco where $C\subset [0,1]$ is an RM-closed set.  Then $F$ attains a maximum value on $C$.
\item Let $(f_{n})_{n\in \N}$ be a Cauchy sequence \(relative to the sup norm\) of continuous functions.  Then the limit function exists and is continuous. 
\item Let $(f_{n})_{n\in \N}$ be a Cauchy sequence \(relative to the sup norm\) of cadlag functions.  Then the limit function exists and is cadlag. 
\end{itemize}
The following theorem should similarly fit the RM of $\ACA_{0}$.
\begin{itemize}
\item The compactness theorem (\cite{thebill}*{Theorem 14.3}) for the Skorohod space (of cadlag functions), which is presented as a generalisation of the Arzel\`a-Ascoli theorem where the latter is part of the RM of $\ACA_{0}$ by \cite{simpson2}*{III.2.9}.   
\item Versions of the Arzel\`a-Ascoli theorem for quasi-continuous functions as in \cite{haloseg2}*{Prop.\ 2.22} and related theorems.
\end{itemize}
Next, we observe again that slight variations of the above principles are not provable from the Big Five.  For instance, the full Jordan decomposition theorem implies $\NIN_{[0,1]}$ by Theorem \ref{tach}.
The same holds for the maximum principle for functions that are usco and regulated (instead of cadlag), as explored in \cite{dagsamXV}.  

\smallskip

In conclusion, we have established the equivalence between $\ACA_{0}$ and some theorems from real analysis.  
We have also shown that slight generalisations or variations are no longer provable in the Big Five and much stronger systems.  
A reasonable explanation of this phenomenon may be found in Section \ref{seor}. 

\subsubsection{Arithmetical transfinite comprehension}
We discuss some equivalences for $\ATR_{0}$ involving real analysis, including restrictions of the Jordan decomposition theorem (Theorem \ref{drd}).  

\smallskip

Now, the RM of $\ATR_{0}$ is a fairly technical affair and this is no different for real analysis.
In particular, we seem to need various instances of the induction axiom, which is not unheard of in second-order RM (\cite{neeman}).
Thus, the base theory becomes somewhat complicated.  For this reason, we merely state that the following items are equivalent to $\ATR_{0}$, assuming 
the base theory $\ACAo$ extended with `enough' induction (see \cite{dagsamXIV}*{\S2.6} for details).
\begin{itemize}
\item For arithmetical formulas $\varphi$ such that 
\be\label{fln2}
(\forall n\in \N)(\exists \textup{ at most one } X\subset \N)\varphi(X, n),
\ee
the set $\{ n\in \N:(\exists X\subset \N)\varphi(X, n)\}$ exists. \label{bont1}
\item For arithmetical $f:[0,1]\di \R$ in $BV$, there is a sequence $(x_{n})_{n\in \N}$ enumerating all points where $f$ is discontinuous.\label{bont2}  
\item For a $\Sigma^1_1$-function  $f:[0,1]\di \R$ in $BV$, there is a sequence $(x_{n})_{n\in \N}$ enumerating all points where $f$ is discontinuous.  \label{bont3}
\item The Jordan decomposition theorem \(Theorem \ref{drd}\) restricted to arithmetical \(or: $\Sigma_{1}^{1}$\) functions in $BV$. \label{bont4}
\item A non-enumerable arithmetical set in $\R$ has a limit point.\label{bont5}
\item Cousin's lemma for codes for Baire 2 functions. 
\item Cousin's lemma for effectively Baire 2 $\Psi:[0,1]\di \R^{+}$.
\item Cousin's lemma for effectively Baire $n$ $\Psi:[0,1]\di \R^{+}$ $(n\geq 2)$.
\end{itemize}
Next, we observe again that slight variations of the above principles are not provable from the Big Five.  For instance, the full Jordan decomposition theorem implies $\NIN_{[0,1]}$ by Theorem \ref{tach}.
The same holds for Cousin's lemma restricted to $BV$-functions, which is proved using $f(x):=\frac{1}{2^{Y(x)+5}}$ from the proof of Theorem \ref{temp3}.

\smallskip

In conclusion, we have established the equivalence between $\ATR_{0}$ and some theorems from real analysis.  
We have also shown that slight generalisations or variations are no longer provable in the Big Five and much stronger systems.  
A reasonable explanation of this phenomenon may be found in Section \ref{seor}. 

\subsubsection{The system $\Pi_{1}^{1}$-comprehension}
We sketch some equivalences for $\FIVE$ involving the supremum principle for effectively Baire 2 functions.  
A slight variation of the latter, namely with `effectively' dropped, is no longer provable in the Big Five and much stronger systems. 

\smallskip

First of all, as to notation, fix $(r_{n})_{n\in \N}$, a standard injective enumeration of the non-negative rational numbers. 
For $B \subset \Q^+$, we say that `$B$ is $\Sigma^1_1$ with parameter $x\in \N^{\N}$', if $A = \{a : r_a \in B\}$ is $\Sigma^1_1$ with parameter $x$. Since we do not always have access to $\Sigma^1_1$-comprehension, we refer to both $A$ and $B$ as (defined) classes.

\smallskip
Secondly, let $\Sigma_{1}^{1}\textsf{-IND}$ be the induction axiom for $\Sigma_{1}^{1}$-formulas.
The proof of the following theorem is relatively involved (see \cite{dagsamXIV}) and we omit it.
\begin{theorem}[$\ACAo$]\label{flapke}
The following are equivalent.
\begin{itemize}
\item For any $x\in \N^{\N}$, any bounded $\Sigma^{1,x}_1$-class in $\Q^+$ has a supremum.\label{birst}
\item An effectively Baire 2 $f:[0,1]\di [0,1]$ has a supremum.\label{becond}
\item For $n \geq 2$, an  effectively Baire $n$ $f:[0,1]\di [0,1]$ has a supremum.\label{bird}
\end{itemize}
Assuming $\QFAC^{0,1}+\Sigma_{1}^{1}\textsf{\textup{-IND}}$, these items are equivalent to $\FIVE$.
\end{theorem}
\noindent
Thirdly, the following theorem suggests that there is a big difference between `Baire 2' and `effectively Baire 2', although these notions are closely related. 
We recall that by Theorem \ref{temp3}, $\Z_{2}^{\omega}+\QFAC^{0,1}$ cannot prove $\NIN_{[0,1]}$.
\begin{thm}[$\RCAo$]\label{tempX}
The following principle implies $\NIN_{[0,1]}$, i.e.\ there is no injection from $[0,1 ]$ to $\N$.
\begin{itemize}
\item A Baire 2 function $f:[0,1]\di [0,1]$ has a supremum$^{\ref{footie}}$.
\end{itemize}
\end{thm}
\begin{proof}
Similar to the proof of Theorem \ref{temp3}.  In particular, any function is Baire 2 in case $\neg \NIN_{[0,1]}$.  
Indeed, if $Y:\R\di \R$ is injective, the following function is readily seen to be Baire 1 for any $f:\R\di \R$:
\[
f_{n}(x):=
\begin{cases}
0 & Y(x)>n\\
f(x) & Y(x)\leq n
\end{cases}.
\]
Clearly, $(f_{n})_{n\in \N}$ converges to $f$ and hence  any function is Baire 2, including $f(x):=\frac{1}{2^{Y(x)+5}}$ used in the proof of Theorem \ref{temp3}.
\end{proof}
In conclusion, the supremum principle for effectively Baire 2 functions is in the RM of $\FIVE$, while the highly similar 
supremum principle for Baire 2 functions is not provable in the Big Five and much stronger systems.  
We discuss a possible explanation in Section \ref{seor}.

\subsection{A form of explanation}\label{seor}
We have observed that many theorems of real analysis are equivalent to the Big Five of RM while slight variations or generalisations are not provable in the Big Five and much stronger systems.  
We provide a possible explanation for this phenomenon in Section \ref{expla}.   We discuss the associated function classes in more detail in Section \ref{closeness}.

\subsubsection{Kleene's observation and beyond}\label{expla}
The following observation goes back to Kleene (\cite{kleenecount}) and underlies the coding of continuous functions in RM (\cite{simpson2}*{II.6.1}).  
\begin{center}
\emph{Continuous functions can be represented via countably much information.}
\end{center}
As it happens, Kleene's observation also holds for many classes of possibly \emph{discontinuous} functions.  
Indeed, the reader is invited to verify that a quasi-continuous or cadlag function $f:[0,1]\di \R$ satisfies the following: the value $f(x)$ for \emph{any} $x\in [0,1]$ is 
determined if we know $f(q)$ for any $q\in [0,1]\cap \Q$, provably in $\ACAo$.   Similarly, the graph of a Baire 1 function $g:[0,1]\di \R$ is determined by $(g_{n}(q))_{n\in \N, q\in \Q}$ if $g$ is the pointwise limit of $(g_{n})_{n\in \N}$, again over $\ACAo$.  

\smallskip

In line with Kleene's centred observation, any member of the aforementioned function classes is determined by countably much information, namely its behaviour on a fixed dense countable set.
As a result, the formula $(\exists x\in [0,1])(f(x)>y)$ is equivalent to an arithmetical formula \emph{only involving second-order parameters}; this equivalence explains why $\RCAo+\WKL_{0}$ or $\RCAo+\ACA_{0}$ suffice to prove the associated theorems, like the supremum principle.  
By contrast, cliquish, usco, and regulated functions are not determined by their behaviour on a fixed dense countable set, as simple examples show.
Moreover, $\Z_{2}^{\omega}+\QFAC^{0,1}$ cannot prove the supremum principle for the latter function classes (Theorem \ref{temp3}).

\smallskip

Things are more complicated in general.  Indeed, for an effectively Baire 2 function $f:[0,1]\di \R$ we only have that `$(\exists x\in [0,1])(f(x)>y)$' is equivalent to a $\Sigma_{1}^{1}$-formula \emph{only involving second-order parameters}.  This equivalence explains why $\ACAo+\FIVE$ proves the associated theorems, like the supremum principle.  
By contrast, for Baire 2 functions $g:[0,1]\di\R$, `$(\exists x\in [0,1])(g(x)>y)$' is in general not equivalent to a second-order formula, say in $\Z_{2}^{\omega}+\QFAC^{0,1}$. 
The latter system does not prove the supremum principle for Baire 2 functions by Theorem \ref{tempX}.

\smallskip

In general, the first group of function classes -including cadlag and effectively Baire 2- comes with 
a kind of second-order `approximation device' while the second class -including regulated and Baire 2- mostly lacks such a device. 
This suggests the following heuristic notion:  we refer to a function class $\Gamma$ as\footnote{The reader will known that the English language boasts the modifier `ish', which means `sort of': a green-ish object is sort of green.} \emph{second-order-ish} if the definition of $\Gamma$ contains additional information, also called an `approximation device', that guarantees that $\Gamma$'s basic third-order properties have equivalent second-order formulations.  
In particular, for second-order-ish $f:[0,1]\di \R$, formulas like `$(\exists x\in [0,1])(f(x)>y)$'  should be equivalent to second-order formulas only involving second-order parameters; this equivalence should be provable in $\Z_{2}^{\omega}+\QFAC^{0,1}$ or (much) weaker systems.

\smallskip

With the gift of hindsight, properties of second-order-ish function classes can be established in $\RCAo$ extended with the Big Five systems, which is one of the main observations of \cite{dagsamXIV}.  
By contrast, basic properties of non-second-order-ish functions can often not be proved in $\Z_{2}^{\omega}$ or even $\Z_{2}^{\omega}+\QFAC^{0,1}$.  In particular, the above equivalences seem to be robust as long as we stay within the second-order-ish function classes, or dually: within the non-second-order-ish ones.  In the next section, we establish that second-order-ish and non-second-order-ish function classes can be very close from the point of view of mathematics. 

\subsubsection{Mathematically versus logical closeness}\label{closeness}
We identify a number of pairs of function classes that are closely related from the point of view of mathematics, but for which only one is second-order-ish. 

\smallskip

First of all, quasi-continuity and cliquishness are closely related notions, not just in the sense of having similar definitions.  
Indeed, cliquish functions are exactly those functions that can be expressed as the sum of two quasi-continuous functions (\cites{quasibor2, malin}).  
The pointwise limit (if it exists) of quasi-continuous functions, is always cliquish (\cite{holausco}*{Cor.\ 2.5.2}).
Cliquish functions are exactly the pointwise discontinuous functions, i.e.\ the final item of Theorem \ref{temp3} is perhaps surprising. 
Nonetheless, the class of quasi-continuous functions is second-order-ish, while the class of cliquish functions is not. 

\smallskip

Secondly, $BV$, cadlag, and regulated functions all have at most countably many points of discontinuity on the reals.   
The cadlag $BV$ functions are known as \emph{normalised} $BV$ and are essential to the Riesz representation theorem (\cite{voordedorst}). 
Nonetheless, the class of cadlag functions is second-order-ish, while the class of regulated (or $BV$) functions is not. 

\smallskip

Thirdly, an usco function $f:[0,1]\di \R$ is Baire 1, which is usually shown by considering the increasing sequence of continuous functions $f_{n}(x):=\inf_{y\in [0,1]} (f(y)+n|x-y|)$.  
By Theorem \ref{temp3}, the latter infimum is generally not available in $\Z_{2}^{\omega}+\QFAC^{0,1}$.  
Moreover, the notions of Baire 1, fragmented, and $B$-measurable function of first class, are equivalent (\cites{leebaire,beren,koumer}).   
Nonetheless, the class of Baire 1 functions is second-order-ish, while the other function classes are not.

\smallskip

Fourth, Baire states in \cite{beren2}*{p.\ 69} that Baire 2 functions can be \emph{represented} by effectively Baire~2 functions.   
The difference lies in the fact that the class of effectively Baire 2 functions is second-order-ish, while the class of Baire 2 functions is not.  Similar observations can be made for the class of (effectively) Baire $n$ functions. 

\section{New Big systems}  \label{fluky2}
\subsection{Introduction}
In this section, we describe the RM-equivalences for four new `Big' systems from real analysis, where the latter are as follows.  
\begin{itemize}
\item The uncountability of $\R$ (Section \ref{Rsec}).
\item The enumeration principle (Section \ref{Esec}).
\item The Baire category theorem (Section \ref{BCTsec}).
\item  Tao's pigeon hole principle for the Lebesgue measure. (Section \ref{Taosec}).
\end{itemize}
The relationships among these principles are sketched in Figure \ref{xxx}. 
While these principles are fundamentally different, the proofs of the associated equivalences follow a rather uniform template, described next.
\begin{temp}\label{tempy}\rm
First of all, in principle, the development of real analysis consists in proving interesting properties of \emph{functions} $f:\R\di\R$ from the axioms of \emph{set} theory. 
Thus, the following continuity and discontinuity sets play a central role.
\begin{center}
$C_{f}=\{x\in \R:\textup{$f$ is continuous at $x$}\}$.\\
$D_{f}=\{x\in \R:\textup{$f$ is discontinuous at $x$}\}$.
\end{center}
In particular, the sets $C_{f}$ and $D_{f}$ allow us to derive properties of functions from properties of sets, i.e.\ the usual direction.

\smallskip

Secondly, the reverse direction, namely deriving properties of sets from properties of functions, shall be done in the below using 
the function $h:[0,1]\di \R$:
\be\label{mopi}
h(x):=
\begin{cases}
0 & x\not \in \cup_{n\in \N}X_{n}\\ 
\frac{1}{2^{n}} & x \in X_{n} \textup{ and $n$ is the least such number}
\end{cases}.
\ee
where $(X_{n})_{n\in \N}$ is a sequence of closed sets.   In particular, $h$ has nice properties in general by Theorem \ref{propi}.  Applying a theorem of real analysis to $h$, one establishes certain properties of the underlying sets $(X_{n})_{n\in \N}$.  
\end{temp}
We note that $h$ is known in real analysis (\cite{myerson}).  Most equivalences sketched in Sections \ref{Rsec}-\ref{Taosec} are established using the previous template or a variation.  

\smallskip

Finally, the above template is of course not completely `turn key': while $h$ is arithmetical in its parameters, the definition of the sets $C_{f}$ and $D_{f}$ involves a quantifier over the reals.  
Moreover, it is not immediate that $h$ belongs to any well-known function class.  These problems all have straightforward solutions inspired by the practice/mainstream of mathematics, as will become clear below. 

\subsection{Properties of countable sets}
\subsubsection{Introduction and preliminaries}\label{horsies}
In this section, we discuss the RM-equivalences for two new `Big' systems, namely the uncountability of $\R$ (Section \ref{Rsec}) and the enumeration principle for countable sets (Section \ref{Esec}).
Clearly, these principles deal with countable sets of reals, a notion that can be formalised in various ways, some of which are listed in Definition \ref{charz}.  
Sets are given by characteristic functions, well-known from measure, probability theory, and second-order RM (\cite{simpson2}*{X.1.12}).  
\bdefi[Sets]\label{charz}~
\begin{itemize}
\item A subset $A\subset \R$ is given by its characteristic function $F_{A}:\R\di \{0,1\}$, i.e.\ we write $x\in A$ for $ F_{A}(x)=1$, for any $x\in \R$.
\item A subset $O\subset \R$ is \emph{open} if $x\in O$ implies $(\exists k\in \N)(B(x, \frac{1}{2^{k}})\subset O)$.
\item A subset $O\subset \R$ is \emph{RM-open} in case there are sequences of reals $(a_{n})_{n\in \N}, (b_{n})_{n\in \N}$ such that $O=\cup_{n\in \N}(a_{n}, b_{n})$.
\item A subset $C\subset \R$ is \emph{closed} if the complement $\R\setminus C$ is open. 
\item A subset $C\subset \R$ is \emph{RM-closed} if the complement $\R\setminus C$ is RM-open. 
\item A set $A\subset \R$ is \emph{enumerable} if there is a sequence of reals that includes all elements of $A$.
\item A set $A\subset \R$ is \emph{countable} if there is $Y: \R\di \N$ that is injective on $A$, i.e.\
\be\label{beg}
(\forall x, y\in A)( Y(x)=_{\N}Y(y)\di x=_{\R}y).  
\ee
\item A set $A\subset \R$ is \emph{strongly countable} if there is a bijection $Y: A\di \N$, i.e.\  $Y:\R\di \N$ satisfies \eqref{beg} and $(\forall n\in \N)(\exists x\in A)(Y(x)=n)$.
\item A set $A\subset \R$ is \emph{finite} if there is $N\in \N$ such that for any finite sequence $(x_{0}, \dots, x_{N})$ of distinct reals, there is $i\leq N$ such that $x_{i}\not \in X$.
\item A set $A\subset \R$ is \emph{height-countable} if there is a \emph{height} function $H:\R\di \N$ for $A$, i.e.\ for all $n\in \N$, $A_{n}:= \{ x\in A: H(x)<n\}$ is finite.  
\item A set $A\subset \R$ is \emph{height-width-countable} if there is a \emph{height} function $H:\R\di \N$ and width function $g\in \N^{\N}$ for $A$, i.e.\ for all $n\in \N$, $A_{n}:= \{ x\in A: H(x)<n\}$ has at most $g(n)$ elements. 
\end{itemize}
\edefi
Regarding height-countability, the notion of `height function' has been studied by Borel (\cite{opborrelen3, opborrelen4, opborrelen5}) and plays a central role in the below, as we discuss next.  
Simply put, height-countability constitutes the most general notion of countability and one readily comes across such sets `in the wild', i.e.\ there is no immediate injection to $\N$, let alone an enumeration.
A good example is provided by {regulated} functions, a notion already studied by Darboux around 1875 in \cite{darb}.  
\begin{exa}[Regulated functions]\label{tekken}\rm
A function $f:[0,1]\di \R$ is regulated if the left and right limits $f(x-)$ and $f(x+)$ exist.  Using the latter, we define the set $D_{f}$ of discontinuity points as
\be\label{drux}
D_{f}:= \{x\in [0,1]:  f(x)\ne_{\R} f(x+)\vee f(x)\ne_{\R} f(x-) \},
\ee
which is arithmetical with $f$ as a parameter.  
Many textbooks establish that $D_{f}$ is countable for regulated $f$ (see e.g.\ \cite{voordedorst, rudin}).  
Nonetheless, one cannot construct an injection from $D_{f}$ to $\N$, let alone an enumeration, in rather strong logical systems.  
By contrast, working in $\ACAo+\QFAC^{0,1}$, one shows that $D_{f}$ is height-countable by considering $D_{f}=\cup_{k\in \N}D_{k}$, where the set
\[\textstyle
D_{k}:= \{x\in [0,1]:  |f(x)- f(x+)|>\frac{1}{2^{k}}\vee |f(x)-f(x-)|>\frac{1}{2^{k}} \}
\]
is finite (in the sense of Definition \ref{horsies}) using a standard compactness argument.  
\end{exa}
Another good example is the smaller class of functions of bounded variation, central to Fourier analysis and going back to Jordan (1895, \cite{jordel}). 
\begin{exa}[Bounded variation]\label{tekken2}\rm
A function $f:[0,1]\di \R$ has bounded variation ($BV$ for short) if if there is $k_{0}\in \N$ such that $k_{0}\geq \sum_{i=0}^{n-1} |f(x_{i})-f(x_{i+1})|$ 
for any partition $x_{0}=0 <x_{1}< \dots< x_{n-1}<x_{n}=1  $.  
For $BV$-functions, the \emph{total variation} is then defined as follows:
\be\label{tomb}\textstyle
V_{0}^{1}(f):=\sup_{0\leq x_{0}< \dots< x_{n}\leq 1}\sum_{i=0}^{n-1} |f(x_{i})-f(x_{i+1})|.
\ee
A $BV$-function is regulated, even in weak logical systems (\cite{dagsamXI}*{Theorem 3.33}), while Weierstrass' monster function is continuous but not in $BV$.
For a $BV$-function, the set $D_{f}$ from \eqref{drux} is height-width countable: we have $D_{f}=\cup_{k\in \N}D_{k}$ where $D_{k}$ has at most $V_{0}^{1}(f)\cdot 2^{k}$-many elements.  
By contrast, one cannot construct an injection from $D_{f}$ to $\N$, let alone an enumeration, in rather strong logical systems.  
\end{exa}
In light of Examples \ref{tekken}-\ref{tekken2}, developing real analysis in a weak system readily gives rise to height-countable sets that do not come with injections to $\N$ or enumerations.  
This suggests that height-countability is the right formalisation of `countable set' for the RM-study of real analysis, as also confirmed by the results in the next sections.     

\smallskip

Finally, Example \ref{tekken} highlights another important aspect of the RM of analysis, namely that we work over $\ACAo$ from Remark \ref{LEM}.
With this base theory, the set $D_{f}$ from \eqref{drux} has its usual definition via a characteristic function; defining this set in $\RCAo$ is much more tricky.
Moreover, much stronger systems (than $\ACAo+\QFAC^{0,1}$) cannot prove the uncountability of the reals, formulated as the statement that the unit interval is not (height-)countable (\cite{dagsamX}).  
The latter is the weakest of the new `Big' systems in this section.   
\subsubsection{The uncountability of the reals}\label{Rsec}
We discuss the RM of the uncountability of the reals formulated as follows, bearing in mind Section \ref{horsies}
\begin{princ}[$\NIN_{\alt}$]
{The unit interval is not height-countable.  }
\end{princ}
We have previously studied the uncountability of $\R$ formulated as `there is no injection from $[0,1]$ to $\N$' in \cite{dagsamX}, but failed to obtain many interesting equivalences. 
By contrast, height-countability readily yields nice equivalences, as follows.

\begin{thm}[$\ACAo+\QFAC^{0,1}$]\label{flonks}
The following are equivalent.
\begin{itemize}
\item A regulated function $f:[0,1]\di \R$ is not totally discontinuous. 
\item The principle $\NIN_{\alt}$.
\end{itemize}
\end{thm}
\begin{proof}
To show that $\NIN_{\alt}$ implies the first item, define $D_{f}$ as in \eqref{drux} and observe that it is height-countable.  
By $\NIN_{\alt}$, there is $x_{0}\in [0,1]\setminus D_{f}$, implying that $f$ is not totally discontinuous.  For the reversal, assume the first item and let $(X_{n})_{n\in \N}$ be a sequence of finite sets in $[0,1]$. 
Now consider the function $h:[0,1]\di \R$ from \eqref{mopi}, which is readily shown to be regulated with $h(x+)=h(x-)=0$ everywhere. 
By the first item, $h$ is continuous at some $x_{0}\in [0,1]$, implying $h(x_{0})=0$ and $x_{0}\not\in \cup_{n\in \N}X_{n}$, i.e.\ $[0,1]\ne \cup_{n}X_{n\in \N}$ for an arbitrary sequence of finite sets, and $\NIN_{\alt}$ follows as required.
\end{proof}
The `reversal' function from \eqref{mopi} is known from the literature (\cite{myerson}), which we only learnt during the writing of \cite{samBIG2}. 
We point out that many variations of the previous result are possible: one can replace `continuous' by `quasi-continuous', or `lower semi-continuous', or `almost continuity' or `the Young condition'.

\smallskip

Next, we list a long list of equivalences; the proof of the previous theorem can be adapted, mostly via minor tricks. 
\begin{thm}[$\ACAo+\QFAC^{0,1}$]\label{flonk}
The following are equivalent.
\begin{enumerate}
\renewcommand{\theenumi}{\alph{enumi}}
\item The uncountability of $\R$ as in $\NIN_{\alt}$.
\item \emph{Volterra's theorem (\cite{volaarde2}) for regulated functions}: there do not exist two regulated functions defined on the unit interval for which the continuity points of one are the discontinuity points of the other, and vice versa.\label{volare1}
\item \emph{Volterra's corollary (\cite{volaarde2}) for regulated functions}: there is no regulated function that is continuous on $\Q\cap[0,1]$ and discontinuous on $[0,1]\setminus\Q$.\label{volare2}
\item For a sequence $(X_{n})_{n\in \N}$ of finite sets in $[0,1]$, the set $[0,1]\setminus \cup_{n\in \N}X_{n}$ is dense \(or: not height countable, or: not countable, or: not strongly countable\).\label{lopi}
\item Any regulated $f:[0,1]\di \R$ is pointwise discontinuous, i.e.\ the set $C_{f}$ is dense in the unit interval. \label{pon3}
\item For regulated $f:[0,1]\di \R$, the set $C_{f}$ is not height countable \(or: not countable, or: not strongly countable, or: not enumerable\). \label{pon35}
\item For regulated $f:[0,1]\di [0,1]$ such that the Riemann integral $\int_{0}^{1}f(x)dx$ exists and is $0$, there is $x\in [0,1]$ with $f(x)=0$. \(Bourbaki, \cite{boereng}*{p.\ 61}\).\label{pon8}
\item For regulated $f:[0,1]\di [0,1]$ such that the Riemann integral $\int_{0}^{1}f(x)dx$ exists and equals $0$, the set $\{x \in [0,1]: f(x)=0\}$ is dense. \(\cite{boereng}*{p.\ 61}\). \label{pon9}
\item Blumberg's theorem \(\cite{bloemeken}\) restricted to regulated functions on $[0,1]$.\label{pon13}
\item For regulated $f:[0,1]\di (0, 1]$, there exist $N\in \N, x\in [0,1]$ such that $(\forall y\in B(x, \frac{1}{2^{N}}))(f(y)\geq \frac{1}{2^{N}})$. \label{pon15}
\item \textsf{\textup{(FTC)}} For regulated $f:[0,1]\di \R$ such that $F(x):=\lambda x.\int_{0}^{x}f(t)dt$ exists, there is $x_{0}\in (0,1)$ where $F(x)$ is differentiable with derivative $f(x_{0})$.\label{ponfar}
\end{enumerate}
\end{thm}
The previous equivalences are robust in that certain restrictions are possible.  
\begin{cor}
The equivalences of Theorem \ref{flonk} still go through if we restrict to any of the following function classes: Baire 1$^{*}$, Baire 2, usco, fragmented, $B$-measurable of first class, symmetrically continuous.  
\end{cor}
The previous equivalences are conceptually pleasing but one can obtain many more by introducing the following principle.  
\begin{princ}[$\NIN_{\alt}'$]\label{drop}
{The unit interval is not height-width-countable.}
\end{princ}
The previous principle is equivalent to $\NIN_{\alt}$ assuming relatively weak principles in the base theory (see \cite{samBIG4} for a list), three of which are as follows.
\begin{itemize}
\item {For any regulated $f:\R\di \R$, there is continuous $g:\R\di \R$ such that $f\leq g$.}
\item {A regulated function $f:[0,1]\di \R$ has bounded Waterman\footnote{The concept of Waterman variation is a generalisation of $BV$ in which the sums in \eqref{tomb} are weighted via a \emph{Waterman sequence} (see \cite{voordedorst} and Remark \ref{essenti}).} variation}.
\item (\textsf{Helly})
Let $(f_{n})_{n\in \N}$ be a sequence of $[0,1]\di [0,1]$-functions in $BV$ with pointwise limit $f:[0,1]\di [0,1]$ which is not in $BV$.  Then there is unbounded $g\in \N^{\N}$ such that $g(n)\leq V_{0}^{1}(f_{n})\leq g(n)+1$ for all $n\in \N$,
\end{itemize}
Moreover, $\NIN_{\alt}'$ is equivalent over $\ACAo$ to most items in Theorems~\ref{flonks} and \ref{flonk} with `regulated' replaced by `$BV$'.
On top of that, $\NIN_{\alt'}$ is equivalent to basic properties of \emph{unordered sums} and basic convergence theorems for Fourier series and Bernstein polynomials (see \cite{samBIG, samBIG3}). 
In the interest of space, we only briefly sketch the results regarding $BV$.

\smallskip

Finally, there are many function classes between the $BV$ and regulated functions, as discussed in the next remark. 
\begin{rem}[Between bounded variation and regulated]\label{essenti}\rm 
The following spaces are intermediate between $BV$ and regulated; all details may be found in \cite{voordedorst}.  

\smallskip

Wiener spaces from mathematical physics (\cite{wiener1}) are based on \emph{$p$-variation}, which amounts to replacing `$ |f(x_{i})-f(x_{i+1})|$' by `$ |f(x_{i})-f(x_{i+1})|^{p}$' in the definition of variation \eqref{tomb}. 
Young (\cite{youngboung}) generalises this to \emph{$\phi$-variation} which instead involves $\phi( |f(x_{i})-f(x_{i+1})|)$ for so-called Young functions $\phi$, yielding the Wiener-Young spaces.  
Perhaps a simpler construct is the Waterman variation (\cite{waterdragen}), which involves $ \lambda_{i}|f(x_{i})-f(x_{i+1})|$ and where $(\lambda_{n})_{n\in \N}$ is a sequence of reals with nice properties; in contrast to $BV$, any continuous function is included in the Waterman space (\cite{voordedorst}*{Prop.\ 2.23}).  Combining ideas from the above, the \emph{Schramm variation} involves $\phi_{i}( |f(x_{i})-f(x_{i+1})|)$ for a sequence $(\phi_{n})_{n\in \N}$ of well-behaved `gauge' functions (\cite{schrammer}).  
As to generality, the union (resp.\ intersection) of all Schramm spaces yields the space of regulated (resp.\ $BV$) functions, while all other aforementioned spaces are Schramm spaces (\cite{voordedorst}*{Prop.\ 2.43 and 2.46}).
In contrast to $BV$ and the Jordan decomposition theorem, these generalised notions of variation have no known `nice' decomposition theorem.  The notion of \emph{Korenblum variation} (\cite{koren}) does have such a theorem (see \cite{voordedorst}*{Prop.\ 2.68}) and involves a distortion function acting on the \emph{partition}, not on the function values (see \cite{voordedorst}*{Def.\ 2.60}).  
\end{rem}
It is no exaggeration to say that there are \emph{many} natural spaces between the regulated and $BV$-functions, all of which yield equivalences for the above.

\subsubsection{The enumeration principle}\label{Esec}
We discuss the RM of the enumeration principle, formulated as follows, bearing in mind Section \ref{horsies}.
This principle is `explosive' in the sense of Theorem \ref{sedr}.  
\begin{princ}[$\enum$]
A height-countable set in $[0,1]$ can be enumerated.   
\end{princ}
Many equivalences for $\enum$ are obtained in \cites{dagsamXI, samBIG, samBIG3, samBIG4}, including the \emph{Jordan decomposition theorem}, which we introduce next. 
Now, the notion of {bounded variation} was first explicitly\footnote{Lakatos in \cite{laktose}*{p.\ 148} claims that Jordan did not invent or introduce the notion of bounded variation in \cite{jordel}, but rather discovered it in Dirichlet's 1829 paper \cite{didi3}.} introduced by Jordan around 1881 (\cite{jordel}) yielding a generalisation of Dirichlet's convergence theorems for Fourier series.  
Indeed, Dirichlet's convergence results are restricted to functions that are continuous except at a finite number of points, while $BV$-functions can have (at most) countable many points of discontinuity, as already studied by Jordan, namely in \cite{jordel}*{p.\ 230}.
We shall make use of the usual definition, as in Example \ref{tekken2}.
Many variations of the Jordan decomposition theorem (Theorem \ref{drd}) are equivalent to $\enum$, some of which are as follows. 
\begin{thm}[$\ACAo+\QFAC^{0,1}$]\label{tach} 
The following are equivalent.
\begin{enumerate}
\renewcommand{\theenumi}{\alph{enumi}}
\item The enumeration principle $\enum$.\label{LA}
\item A non-enumerable and closed set in $\R$ has a limit point.  
\item For regulated $f:[0,1]\di \R$, the set $D_{f}$ is enumerable.\label{LB}
\item For regulated $f:[0,1]\di \R$ and $p, q\in [0,1]\cap \Q$, $\sup_{x\in [p,q]}f(x)$ exists\footnote{To be absolutely clear, we assume the existence of a `supremum operator' $\Phi:\Q^{2}\di \R$ such that $\Phi(p, q)=\sup_{x\in [p, q]}f(x)$ for all $p, q\in [0,1]\cap \Q$.  For Baire 1 functions, this kind of operator exists in $\ACAo$ by \cite{dagsamXIV}*{\S2}, even for irrational intervals.}.\label{LBQ}  
\item For regulated and pointwise discontinuous $f:[0,1]\di \R$, $D_{f}$ is enumerable.\label{LBP}  
\item For regulated $f:[0,1]\di \R$, there is $(x_{n})_{n\in \N}$ enumerating $[0,1]\setminus B_{f}$, where 
\be\label{BF}
B_{f}:=\{ x\in [0,1]: f(x)=\lim_{n\di \infty} B_{n}(f, x)\}
\ee
and $B_{n}(f)$ is the $n$-th Bernstein polynomial.\label{LC}  
\item \(Jordan\) For $f:\R\di \R$ which is in $BV([0, a])$ for all $a>0$, there are monotone $g, h:\R\di \R$ such that $f(x)=g(x)-h(x)$ for $x\geq 0$.\label{LD}
\item The combination of: \label{LE}
\begin{itemize}
\item[(g.1)] Helly's selection theorem as formulated just below Principle \ref{drop}.
\item[(g.2)] \(Jordan\) For $f:[0,1]\di \R$ in $BV$, there are non-decreasing $g, h:\R\di \R$ such that $f(x)=g(x)-h(x)$ for $x\in [0,1]$.
\end{itemize}
\item The combination of: \label{LEZZ}
\begin{itemize}
\item[(h.1)] Helly's selection theorem as formulated just below Principle \ref{drop}.
\item[(h.2)] The principle $\enum'$: any height-width countable set is enumerable. 
\end{itemize}
\end{enumerate}
\end{thm}
\begin{proof}
First of all, $\enum$ implies item \eqref{LB} as $D_{f}$ is height-countable for regulated $f:[0,1]\di \R$.  
The enumeration of $D_{f}$ then yields most other items as e.g.\ the suprema in \eqref{tomb} and $\sup_{x\in [p, q]}f(x)$ can be replaced
by suprema over $\Q$ and the enumeration of $D_{f}$.  

\smallskip

Secondly, the function $h:[0,1]\di \R$ as in \eqref{mopi} is regulated if the underlying sets $X_{n}$ are finite.  
An enumeration of $D_{h}$ readily provides an enumeration of $\cup_{n\in \N}X_{n}$, i.e.\ $\enum$ is thus obtained from item \eqref{LB}. 
\end{proof}
\begin{cor}
One can replace `monotone' in items \eqref{LD} or \eqref{LE} by:
\begin{itemize}
\item $U_{0}$-function, or:
\item regulated $f:[0,1]\di \R$ such that for all $x\in (0,1)$, we have
\be\label{fereng}\textstyle
|f(x)-\lim_{n\di \infty} B_{n}(f, x)|\leq |\frac{f(x+)-f(x-)}{2}|. 
\ee
\end{itemize}
\end{cor}
As discussed in \cite{samBIG3, samBIG4}, a lot of variations of the previous equivalences are possible, including the restriction to Riemann integrable functions or the convergence of Fourier series or Hermit-Fejer polynomials. 

\smallskip

Next, the previous equivalences involve theorems of real analysis that do not mention countable sets. Of course, $\enum$ is equivalent to basic properties of height-countable sets, a selection of which is as follows.  
\begin{itemize}
\item A height-countable set in the unit interval has a supremum.
\item A height-countable linear ordering $(X, \preceq_{X})$ for $X\subset \R$ is order-isomorphic to a subset of $\Q$. 
\item A height-countable and dense linear ordering without endpoints $(X, \preceq_{X})$ for $X\subset \R$ is order-isomorphic to $\Q$. 
\item For height-countable well-orders {$(X, \preceq_{X}) $ and $ (Y, \preceq_{Y})$} where $X, Y\subset \R$, the former order is order-isomorphic to the latter order or an initial segment of the latter order, or vice versa.
\end{itemize}
Unfortunately, the notion of height-countability was only identified after the completion of \cite{dagsamXI}, i.e.\ the latter is formulated using injections to $\N$, which causes some additional technical overhead. 

\smallskip

Finally, the principle $\enum$ is \emph{explosive} in that combining it with certain comprehension functionals yields much stronger comprehension axioms.
The results in \cite{kruisje} suggest that $\ACAo+\NIN_{\alt}$ is conservative over $\ACA_{0}$.
\begin{thm}\label{sedr}~
\begin{itemize}
\item The system $\ACAo+\enum$ proves $\ATR_{0}$.
\item The system\footnote{The system $\FIVE^{\omega}$ is $\RCAo+(\SS^{2})$, which is $\Pi_{3}^{1}$-conservative over $\FIVE$ (\cite{yamayamaharehare}) and where
\be\tag{$\SS^{2}$}
(\exists\SS^{2}:\N^{\N}\di \{0,1\})(\forall f \in \N^{\N})\big[  (\exists g \in \N^{\N})(\forall n\in \N)(f(\overline{g}n)=0)\asa \SS(f)=0  \big].
\ee
The functional $\SS^{2}$ is also called the \emph{Suslin} functional.} $\FIVE^{\omega}+\enum$ proves $\SIX$.
\end{itemize}
\end{thm}
\begin{proof}
The first item is straightforward in light of the well-known `unique existence' version of $\ATR_{0}$ from \cite{simpson2}*{V.5.2}.
The second item follows in the same way since $\FIVE$ proves uniformisation for $\Pi_{1}^{1}$-formulas (\cite{simpson2}*{VI.2.1}).
\end{proof}
Regarding the second item, at least from the point of ordinal analysis, $\SIX$ is said to be much stronger than $\FIVE$ (\cite{rathjenICM, loefenlei}).

\subsubsection{An aside: hyperarithmetical analysis}\label{HYPsec}
We have observed that $\NIN_{\alt}$ and $\enum$ boast many equivalences, i.e.\ constitute new `Big' systems.  
These principles are fundamentally based on \emph{height functions} while with more technical overhead, the same equivalences 
could be obtained using \emph{injections to $\N$}, which is actually done in \cite{dagsamXI}.  The question remains what
the role of strongly countable sets and bijections to $\N$ is.  We provide a brief answer to this question in this section.  

\smallskip

First of all, the term \emph{hyperarithmetical analysis} refers to a cluster of logical systems just beyond $\ACAo$, including $\SAC$ and $\WSAC$ (see e.g.\ \cite{simpson2, monta2} for definitions). 
A logical system that is sandwiched between two systems of hyperarithmetical analysis, is also a system of hyperarithmetical analysis.  Now, $\ACAo+\QFAC^{0,1}$ is conservative over $\SAC$ by \cite{hunterphd}*{Cor.\ 2.6}.  
Hence, it seems reasonable to say that \emph{a higher-order system $T$ inhabits the range of hyperarithmetical analysis} in case $T$ satisfies the following:
\be\label{zoiu}
\ACAo+\QFAC^{0,1}\di T\di \WSAC.
\ee
Secondly, $\ACAo+\enum''$ exists in the range of hyperarithmetical analysis. 
\begin{princ}[$\enum''$]
Any strongly countable set $A\subset [0,1]$ can be enumerated.
\end{princ}
A number of equivalences and related results can be found in \cites{dagsamXI, samHYP, samcount}.  In particular, basic properties of (Lipschitz) continuous functions on compact metric spaces, $BV$-functions, and K\"onig's (original) infinity lemma are shown to inhabit the range of hyperarithmetical analysis, i.e.\ they behave like $T$ in \eqref{zoiu}.

\smallskip

In conclusion, height functions and `injections to $\N$' play a central role in the RM of the uncountability of $\R$ and the Jordan decomposition theorem.  
The stronger notion of `bijections to $\N$' yields interesting principles that inhabit the range of hyperarithmetical analysis as in \eqref{zoiu}.

\subsection{Measure and category}
\subsubsection{Introduction}
In this section, we discuss the RM-equivalences for two new `Big' systems, namely the \emph{Baire category theorem} (Section \ref{BCTsec}) and \emph{Tao's pigeon hole principle for the Lebesgue measure} (Section \ref{Taosec}).
This RM-study involves basic properties of usco and cliquish functions. 
While measure and category are fundamentally different, the proofs are often surprisingly similar.

\smallskip

First of all, the following definitions are essential.  
\bdefi~
\begin{itemize}
\item A set $A\subset \R$ is \emph{dense} in $B\subset \R$ if $(\forall k\in \N,b\in B)(\exists a\in A)(|a-b|<\frac{1}{2^{k}})$.
\item A set $A\subset \R$ is \emph{nowhere dense} in $B\subset \R$ if $A$ is not dense in any open sub-interval of $B$.  
\item A set $A\subset \R$ is \emph{measure zero} if for any $\eps>0$ there is a sequence of open intervals $(I_{n})_{n\in \N}$ such that $\cup_{n\in \N}I_{n}$ covers $A$ and $\eps>\sum_{n=0}^{\infty}|I_{n}|$. 
\end{itemize}
\edefi
Secondly, we recall that the sets $C_{f}$ and $D_{f}$ play a central role in real analysis while the definition of continuity involves quantifiers over $\R$.  
For regulated functions, this problem was solved `by definition', namely using \eqref{drux}.
Nonetheless, for non-regulated functions, the sets $C_{f}$ and $D_{f}$ can generally not be defined in fairly strong logical systems.  
A rather elegant solution is provided by \emph{oscillation functions} as in Definition \ref{oscfn}.  We note that Ascoli, Riemann, and Hankel already considered the notion of oscillation in the context of Riemann integration (\cites{hankelwoot, rieal, ascoli1}).  
\bdefi[Oscillation functions]\label{oscfn}
For any $f:\R\di \R$, the associated \emph{oscillation functions} are defined as follows: $\osc_{f}([a,b]):= \sup _{{x\in [a,b]}}f(x)-\inf _{{x\in [a,b]}}f(x)$ and $\osc_{f}(x):=\lim _{k \di \infty }\osc_{f}(B(x, \frac{1}{2^{k}}) ).$
\edefi
We  stress that $\osc_{f}:\R\di \R$ is \textbf{only}\footnote{To be absolutely clear, the notation `$\osc_{f}$' and the appearance of $f$ therein in particular, is purely symbolic and does not include lambda abstraction involving `$\lambda f$'.} a third-order function, as clearly indicated by its type.   On a related technical note, while the suprema, infima, and limits in Definition~\ref{oscfn} do not always exist in weak systems, formulas like $\osc_{f}(x)>y$ \emph{always} make sense as shorthand 
 for the standard definition of the suprema, infima, and limits involved; this `virtual' or `comparative' meaning is part and parcel of (second-order) RM in light of \cite{simpson2}*{X.1}.

\smallskip

Finally, $h:[0,1]\di \R$ from \eqref{mopi} has very nice properties proved in \cite{samBIG2}*{\S1.3.4}.  
\begin{thm}[$\ACAo$]\label{propi} Let $(X_{n})_{n\in \N}$ be an increasing sequence of closed sets.
\begin{itemize} 
\item The function $h:[0,1]\di \R$ from \eqref{mopi} is usco and cliquish.  
\item If each $X_{n}$ is also nowhere dense, then $h:[0,1]\di \R$ from \eqref{mopi} is its own oscillation function. 
\end{itemize}
\end{thm}
As it happens, functions that are their own oscillation function have been studied in the mathematical literature (see e.g.\ \cite{kosten}). 

\subsubsection{The Baire category theorem}\label{BCTsec}
We discuss the RM of the Baire category theorem, denoted $\BCT_{[0,1]}$, which involves basic properties of usco and cliquish functions, as well as intermediate classes (\cites{samBIG2, dagsamVII}). 
\begin{thm}[$\BCT_{[0,1]}$]\label{konkli}
If $ (O_n)_{n \in \N}$ is a sequence of dense open sets of reals, then 
$ \bigcap_{n \in\N } O_n$ is non-empty.
\end{thm}
The Baire category theorem for the real line was first proved by Osgood (\cite{fosgood}) and later by Baire (\cite{beren2}) in a more general setting.
\begin{thm}[$\ACAo$]\label{clockn} The following are equivalent to $\BCT_{[0,1]}$. 
\begin{enumerate}
 \renewcommand{\theenumi}{\alph{enumi}}
\item For usco $f:[0,1]\di \R$, there is a point $x\in [0,1]$ where $f$ is continuous.\label{bctii}
\item For usco and countably continuous $f:[0,1]\di \R$, there is a point $x\in [0,1]$ where $f$ is continuous.\label{bctv}
\item Any usco $f:[0,1]\di \R$ is bounded below on some interval, i.e.\ there exist $q\in \Q$ and $c, d\in [0,1]$ such that $(\forall y\in (c, d))(f(y)\geq q)$.\label{bct6}
\item For usco $f:[0,1]\di \R^{+}$, there are $c, d\in [0,1]$ with $0<\inf_{x\in [c, d]}f(x)$.\label{bct7}
\item For Baire $1^{*}$ $f:[0,1]\di [0,1]$, there is $x\in [0,1]$ where $f$ is continuous.\label{b12}
\item \emph{(Uniform boundedness principle \cite{bro})} A sequence of lsco functions that is pointwise bounded on $[0,1]$, is uniformly bounded on some $(c, d)\subset [0,1]$.\label{franken}
\item For cliquish $f:[0,1]\di \R$ which has an oscillation function $\osc_{f}:[0,1]\di \R$, there is a point $x\in [0,1]$ where $f$ is continuous. \label{la2}
\item For cliquish $f:[0,1]\di \R$ which has an oscillation function $\osc_{f}:[0,1]\di \R$, there is a point $x\in [0,1]$ where $f$ is lqco. \label{la3}
\item For cliquish and uqco $f:[0,1]\di \R$ which has an oscillation function $\osc_{f}:[0,1]\di \R$, there is a point $x\in [0,1]$ where $f$ is continuous. \label{la4}
\item For totally discontinuous $f:[0,1]\di \R$ with oscillation function $\osc_{f}:[0, 1]\di \R$, there is $N\in \N$ and $c, d\in  [0,1]$ with $\osc_{f}(x)\geq \frac{1}{2^{N}}$ for $x\in (c, d)$.\label{la5}
\item \emph{Volterra's theorem}: there do not exist two \emph{pointwise discontinuous} $f, g:[0,1]\di \R$ with associated oscillation functions for which the continuity points of one are the discontinuity points of the other, and vice versa.\label{volare1337}
\item \emph{Volterra's theorem} for `pointwise discont.' replaced by `cliquish' or `usco'.\label{volare23}
\item \emph{Volterra's corollary}: there is no function with an oscillation function that is continuous on $\Q\cap[0,1]$ and discontinuous on $[0,1]\setminus\Q$.\label{volare4}
\item \emph{Volterra's corollary} restricted to usco functions.\label{volare3}
\item Blumberg's theorem \(\cite{bloemeken}\) restricted to usco \(or cliquish with an oscillation function\) functions on $[0,1]$.\label{volare6}
\item For usco $f:[0,1]\di \R$, the set $ B_{f}$ from \eqref{BF} is non-empty.\label{bao1}
\item For cliquish $f:[0,1]\di \R$ with an oscillation function, $ B_{f}$ is non-empty.\label{bao2}
\end{enumerate}
\end{thm}
\begin{proof}
We establish one `textbook' equivalence. 
The usual proof establishes item~\eqref{la2} using $\BCT_{[0,1]}$: let $f:[0,1]\di \R$ be cliquish with oscillation function $\osc_{f}:[0,1]\di \R$.  
Then $D_{f}=\{x\in [0,1]: \osc_{f}(x)>0 \}$ and $D_{f}=\cup_{k\in \N}D_{k}$ where $D_{k}=\{x\in [0,1]:\osc_{f}(x)\geq \frac{1}{2^{k}}\} $ and the latter is closed and nowhere dense. 
By $\BCT_{[0,1]}$, we have $[0,1]\ne \cup_{n\in \N}D_{k}=D_{f}$, i.e.\ $C_{f}\ne \emptyset$ as required for item \eqref{la2}.

\smallskip

For the reversal, assume item \eqref{la2}, let $(O_{n})_{n\in \N}$ be a sequence of open and dense sets, and consider $h: [0,1]\di \R$ from \eqref{mopi} where $X_{n}:= [0,1]\setminus O_{n}$.  
By Theorem~\ref{propi}, we may apply item \eqref{la2} to $h$, i.e.\ there is $x_{0}\in C_{h}$.  By definition, we have $h(x_{0})=0$ and hence $x_{0}\in \cap_{n\in \N}O_{n}$, as required for $\BCT_{[0,1]}$
\end{proof}
The equivalence for e.g.\ item \eqref{bctii} is much harder as no oscillation function is assumed. 
The approach in \cite{samBIG2}*{\S2.2} is to take a fairly `textbook proof' and observe that (for usco functions only) all essential objects in the proof have an arithmetical definition, i.e.\ the proof goes through over $\ACAo+\BCT_{[0,1]}$. 
The following items are equivalent to $\BCT_{[0,1]}$ assuming extra induction as in $\IND_{\R}$ right below.
\begin{itemize}
\item For fragmented $f:[0,1]\di \R$ which has an oscillation function $\osc_{f}:[0,1]\di \R$, there is a point $x\in [0,1]$ where $f$ is continuous. \label{fra2}
\item For $B$-measurable of class 1 and cliquish $f:[0,1]\di \R$ which has an oscillation function $\osc_{f}:[0,1]\di \R$, there is $x\in [0,1]$ where $f$ is continuous.\label{fra3}
\end{itemize}
\begin{princ}[$\IND_{\R}$]
For $F:(\R\times \N)\di \N, k\in \N$, there is $X\subset \N$ such that
\[
(\forall n\leq k)\big[ (\exists x\in \R)(F(x, n)=0)\asa n\in X\big].
\]
\end{princ}
We believe many variations of the above are possible as there are many function classes between the usco and cliquish functions. 

\subsubsection{Tao's pigeon hole principle}\label{Taosec}
We discuss the RM of the Tao's pigeon hole principle for the Lebesgue measure, denoted $\PHP_{[0,1]}$.  This involves basic properties of the Riemann integral and restrictions to certain function classes (\cites{samBIG2, samBIG4}). 
\begin{princ}[$\PHP_{[0,1]}$]\label{PHP}
If $ (X_n)_{n \in \N}$ is an increasing sequence of measure zero closed sets of reals in $[0,1]$, then $ \bigcup_{n \in\N } X_n$ has measure zero.
\end{princ} 
The following theorem \emph{almost} establishes the essential part of the \emph{Vitali-Lebesgue theorem}, i.e.\ that a Riemann integrable function is continuous ae. 
\begin{thm}[$\ACAo+\IND_{\R}$]\label{sepa}
For Riemann integrable $f:[0,1]\di \R$ with oscillation function $\osc_{f}$, $D_{k}:=\{x\in [0,1]:\osc_{f}(x)\geq \frac{1}{2^{k}}\} $ has measure zero.
\end{thm}
\begin{proof}
The standard proof-by-contradiction formalises (see \cite{samBIG2}*{Theorem 3.3.}).
\end{proof}
\noindent
We can now connect $\PHP_{\R}$ and the Vitali-Lebesgue theorem.  
\begin{thm}[$\ACAo+\IND_{\R}+\QFAC^{0,1}$]\label{duck555}
The following are equivalent.  
\begin{enumerate}
\renewcommand{\theenumi}{\alph{enumi}}
\item The pigeonhole principle for measure spaces as in $\PHP_{[0,1]}$.
\item \(Vitali-Lebesgue\) For Riemann integrable $f:[0,1]\di \R$ with an oscillation function, the set $D_{f}$ has measure $0$.\label{tao1}
\item For Riemann integrable usco $f:[0,1]\di \R$, the set $D_{f}$ has measure $0$.\label{tao2}
\item For Riemann integrable $f:[0,1]\di [0,1]$ with an oscillation function and $\int_{0}^{1}f(x)dx=0$, the set $\{x\in [0,1]:f(x)=0\}$ has measure one.\label{tao3}
\item For Riemann integrable usco $f:[0,1]\di [0,1]$ with $\int_{0}^{1}f(x)dx=0$, the set $\{x\in [0,1]:f(x)=0\}$ has measure one.\label{tao4}
\item \textsf{\textup{(FTC)}} For Riemann integrable $f:[0,1]\di \R$ with an oscillation function and $F(x):=\lambda x.\int_{0}^{x}f(t)dt$, the following set exists: 
\be\label{tanko}
\{x\in [0,1]:F \textup{ is differentiable at $x$ with derivative } f(x)\}
\ee
and has measure one.\label{tao5}
\item \textsf{\textup{(FTC)}} The previous item for usco functions.\label{tao6}
\item For $f:[0,1]\di \R$ not continuous almost everywhere with oscillation function $\osc_{f}:[0, 1]\di \R$, there is $N\in \N$ and $E\subset  [0,1]$ of positive measure such that $\osc_{f}(x)\geq \frac{1}{2^{N}}$ for $x\in E$.\label{topanga2}
\end{enumerate}
We can replace `usco' by `cliquish with an oscillation function' in the above.  
\end{thm}
\begin{proof}
Most items follow from $\PHP_{[0,1]}$ by Theorem \ref{sepa}.
To derive $\PHP_{[0,1]}$ from item \eqref{tao2}, let $(X_{n})_{n\in \N}$ be a sequence of measure zero sets and note that $h:[0,1]\di \R$ from \eqref{mopi} is usco by Theorem \ref{propi}. 
One readily verifies that $h$ is Riemann integrable with $\int_{0}^{1}h(x)dx=0$.  By item \eqref{tao2}, $D_{f}=\cup_{n\in \N}X_{n}$ has measure $0$. 
\end{proof}
While $\enum$ is explosive by Theorem \ref{sedr}, $\PHP_{[0,1]}$ seems rather tame, as follows.
\begin{thm}\label{weng} 
The system $\ACAo+\PHP_{[0,1]}$ is $\Pi_{2}^{1}$-conservative over $\ACA_{0}$.
\end{thm}
\begin{proof}
Use the conservation result for $\ACAo$ from \cite{kruisje}.
\end{proof}
\noindent
By Theorem \ref{sedr}, we have that $\ACAo+\PHP_{[0,1]}$ cannot prove $\enum$.
We note that $\PHP_{[0,1]}$ is provable in $\ZF$ and much weaker systems by \cite{samBIG2}*{Theorem 3.6}.

\begin{ack}\rm 
We thank Anil Nerode for his valuable advice, especially the suggestion of studying nsc for the Riemann integral, and discussions related to Baire classes.
We thank Ulrich Kohlenbach for (strongly) nudging us towards Theorem \ref{nudge} as part of the second author's Habilitation thesis (\cite{samhabil}) at TU Darmstadt.
We thank Dave L.\ Renfro for his efforts in providing a most encyclopedic summary of analysis, to be found online.  
Our research was supported by the \emph{Deutsche Forschungsgemeinschaft} via the DFG grant SA3418/1-1 and the \emph{Klaus Tschira Boost Fund} via the grant Projekt KT43 .
We express our gratitude towards these institutions.    
\end{ack}

\appendix

\section{Definitions}
\subsection{Function classes}\label{FCS}
We collect most of the definitions used in the above.
These are taken from the literature and collected here for reference. 
\bdefi\label{flung} 
For $f:[0,1]\di \R$, we have the following definitions:
\begin{itemize}
\item $f$ is \emph{upper semi-continuous} at $x_{0}\in [0,1]$ if $f(x_{0})\geq_{\R}\lim\sup_{x\di x_{0}} f(x)$,
\item $f$ is \emph{lower semi-continuous} at $x_{0}\in [0,1]$ if $f(x_{0})\leq_{\R}\lim\inf_{x\di x_{0}} f(x)$,
\item $f$ is \emph{quasi-continuous} at $x_{0}\in [0, 1]$ if for $ \epsilon > 0$ and an open neighbourhood $U$ of $x_{0}$, 
there is a non-empty open ${ G\subset U}$ with $(\forall x\in G) (|f(x_{0})-f(x)|<\eps)$.
\item $f$ is \emph{cliquish} at $x_{0}\in [0, 1]$ if for $ \epsilon > 0$ and an open neighbourhood $U$ of $x_{0}$, 
there is a non-empty open ${ G\subset U}$ with $(\forall y, z\in G) (|f(y)-f(z)|<\eps)$.
\item $f$ is \emph{regulated} if for every $x_{0}$ in the domain, the `left' and `right' limit $f(x_{0}-)=\lim_{x\di x_{0}-}f(x)$ and $f(x_{0}+)=\lim_{x\di x_{0}+}f(x)$ exist.  
\item $f$ has bounded variation if if there is $k_{0}\in \N$ such that $k_{0}\geq \sum_{i=0}^{n-1} |f(x_{i})-f(x_{i+1})|$ for any partition $x_{0}=0 <x_{1}< \dots< x_{n-1}<x_{n}=1  $. 
\item $f$ is \emph{c\`adl\`ag} if it is regulated and $f(x)=f(x+)$ for $x\in [0,1)$.
\item $f$ is a $U_{0}$-function \(\cite{jojo, lorentzg, gofer2, voordedorst}\) if it is regulated and for all $x\in (0,1)$: 
\be\label{zolk}
\min(f(x+), f(x-))\leq f(x)\leq \max(f(x+), f(x-)),
\ee
\item $f$ is \emph{Darboux} if it has the intermediate value property, i.e.\ if $a, b\in [0,1], c\in \R$ are such that $a\leq b$ and either $f(a)\leq c\leq f(b)$ or $f(b)\leq c\leq f(a)$, then there is $d\in [a,b]$ with $f(d)=c$.
\item $f$ is \emph{Baire 0} if it is a continuous function. 
\item $f$ is \emph{Baire $n+1$} if it is the pointwise limit of a sequence of Baire $n$ functions.
\item $f$ is \emph{effectively Baire $n$} $(n\geq 2)$ if there is a sequence $(f_{m_{1}, \dots, m_{n}})_{m_{1}, \dots, m_{n}\in \N}$ of continuous functions such that for all $x\in [0,1]$, we have 
\[\textstyle
f(x)=\lim_{m_{1}\di \infty}\lim_{m_{2}\di \infty}\dots \lim_{m_{n}\di \infty}f_{m_{1},\dots ,m_{n}}(x).
\]
\item $f$ is \emph{Baire 1$^{*}$} if\footnote{The notion of Baire 1$^{*}$ goes back to \cite{ellis} and equivalent definitions may be found in \cite{kerkje}.  
In particular,  Baire 1$^{*}$ is equivalent to the Jayne-Rogers notion of \emph{piecewise continuity} from \cite{JR}.} there is a sequence of closed sets $(C_{n})_{n\in \N}$ such $[0,1]=\cup_{n\in \N}C_{n}$ and $f_{\upharpoonright C_{m}}$ is continuous for all $m\in \N$.
\item $f$ is \emph{fragmented} if for any $\eps>0$ and closed $C\subset [0,1]$, there is non-empty relatively\footnote{For $A\subseteq B\subset \R$, we say that \emph{$A$ is relatively open \(in $B$\)} if for any $a\in A$, there is $N\in \N$ such that $B(x, \frac{1}{2^{N}})\cap B\subset A$.  Note that $B$ is always relatively open in itself.} open $O\subset C$ such that $\textup{\textsf{diam}}(f(O))<\eps$.
\item $f$ is \emph{$B$-measurable of class 1} if for every open $Y\subset \R$, the set $f^{-1}(Y)$ is $\F_{\sigma}$, i.e.\ a union over $\N$ of closed sets. 
\item $f$ is \emph{continuous almost everywhere} if it is continuous at all $x\in [0,1]\setminus E$, where $E$ is a measure zero set.
\item $f$ is \emph{pointwise discontinuous} if for any $x\in [0,1]$ and $\eps>0$, there is $y\in [0,1]$ such that $f$ is continuous at $y$ and $|x-y|<\eps$ \(Hankel, 1870, \cite{hankelwoot}\).  
\item $f$ is \emph{almost continuous} \(Husain, see \cites{husain,bloemeken}\) at $x\in [0,1]$ if for any open $G\subset \R$ containing $f(x)$, the set $\overline{f^{-1}(G)}$ is a neighbourhood of $x$,
\item $f$ has the \emph{Young condition} at $x\in [0,1]$ if there are sequences $(x_{n})_{n\in \N}, (y_{n})_{n\in \N}$ on the left and right of $x$ with the latter as limit and $\lim_{n\di \infty}f(x_{n})=f(x)=\lim_{n\di \infty}f(y_{n})$.
\item $f$ is \emph{cliquish} at $x_{0}\in [0, 1]$ if for $ \epsilon > 0$ and any open neighbourhood $U$ of $x_{0}$, 
there is a non-empty open ${ G\subset U}$ with $(\forall x, y\in G) (|f(x)-f(y)|<\eps)$,
\item $f$ is \emph{upper \(resp.\ lower\) quasi-continuous} at $x_{0}\in [0, 1]$ if for $ \epsilon > 0$ and any open neighbourhood $U$ of $x_{0}$, there is a non-empty open ${ G\subset U}$ with $(\forall x\in G) (f(x)< f(x_{0})+\eps)$ \(resp.\  $(\forall x\in G) (f(x)> f(x_{0})-\eps)$\),
\item $f$ is \emph{countably continuous}\footnote{The notion of countably discontinuity, under a different name, goes back to Lusin (see \cite{novady}).} if there is a sequence $(E_{n})_{n\in \N}$ such $[0,1]=\cup_{n\in \N}E_{n}$ and $f_{\upharpoonright E_{m}}$ is continuous for all $m\in \N$.
\end{itemize}
\edefi
As to notations, a common abbreviation is `usco' and `lsco' for the first two items, while one often just writes `cadlag', i.e.\ without the accents.  
For the penultimate item, `uqco' (resp.\ lqco) is an abbreviation for upper (resp.\ lower) quasi-continuity. 
Moreover, if a function has a certain weak continuity property at all reals in $[0,1]$ (or its intended domain), we say that the function has that property.  

\smallskip

Regarding the notion of `effectively Baire $n$' in Definition \ref{flung}, the latter is used, using codes for continuous functions, in second-order RM (see \cite{basket, basket2}).

\bye